\documentclass[11pt,english]{article}

\usepackage[margin = 2.5cm]{geometry}

\usepackage{amsthm}
\usepackage{amsmath}
\usepackage{amssymb}
\usepackage{setspace}
\usepackage{mathtools}
\usepackage{graphicx}
\usepackage[hidelinks]{hyperref}
\usepackage{thm-restate}
\usepackage{cleveref}
\usepackage{enumitem}
\setlist[itemize]{leftmargin=*} %sets itemize indentation to 0
\usepackage{framed}
\usepackage{subcaption}
\usepackage{bbm}

\providecommand{\noopsort}[1]{}

\usepackage{floatrow}

\theoremstyle{plain}

\newtheorem*{thm*}{Theorem}
\newtheorem{thm}{Theorem}
\Crefname{thm}{Theorem}{Theorems}

\newtheorem*{lem*}{Lemma}
\newtheorem{lem}[thm]{Lemma}
\Crefname{lem}{Lemma}{Lemmas}

\newtheorem*{claim*}{Claim}
\newtheorem{claim}[thm]{Claim}
\crefname{claim}{Claim}{Claims}
\Crefname{claim}{Claim}{Claims}

\newtheorem{prop}[thm]{Proposition}
\Crefname{prop}{Proposition}{Propositions}

\crefname{cor}{Corollary}{Corollaries}

\newtheorem{conj}[thm]{Conjecture}
\crefname{conj}{Conjecture}{Conjectures}

\Crefname{qn}{Question}{Questions}

\Crefname{obs}{Observation}{Observations}

\newtheorem{ex}[thm]{Example}
\Crefname{ex}{Example}{Examples}

\theoremstyle{definition}

\Crefname{prob}{Problem}{Problems}

\Crefname{defn}{Definition}{Definitions}

\theoremstyle{remark}
\newtheorem*{rem}{Remark}

\newcommand{\remove}[1]{}

\newcommand{\sat}{saturated}

\newcommand{\eps}{\varepsilon}
\newcommand{\sm}{\setminus}
\renewcommand{\ss}{\subseteq}

\newcommand{\ups}[2]{[#1]^{(#2)}}
\newcommand{\crdep}{cross dependant}
\newcommand{\crsat}{cross saturated}
\newcommand{\boxpr}{\,\square\,}

\newcommand{\setcomp}[1]{\overline{#1}}		%[n] \sm #1}
\newcommand{\pwcomp}[1]{\overline{#1}}
\newcommand{\comp}[1]{#1\:\!^C}	%{\mathsf{c}}}
\newcommand{\dbcomp}[1]{\pwcomp{\comp{#1}}}

\newcommand{\A}{\mathcal{A}}
\newcommand{\B}{\mathcal{B}}
\newcommand{\C}{\mathcal{C}}

\newcommand{\F}{\mathcal{F}}
\newcommand{\G}{\mathcal{G}}

\newcommand{\R}{\mathbb{R}}

\newcommand{\ang}[1]{\langle #1 \rangle}
\newcommand{\abs}[1]{\left| #1 \right| }

\newcommand{\psn}[1]{2^{[#1]}}
\newcommand{\ps}{\psn{n}}   %power set

\newcommand*{\defeq}{\mathrel{\vcenter{\baselineskip0.5ex \lineskiplimit0pt
                     \hbox{\scriptsize.}\hbox{\scriptsize.}}}%
                     =}

\begin{document}

%\allowdisplaybreaks

\title{\vspace{-0.85cm} Minimum saturated families of sets}

\author{
	Matija Buci\'c\thanks{
	    Department of Mathematics, 
	    ETH, 
	    8092 Zurich;
	    e-mail: \texttt{matija.bucic}@\texttt{math.ethz.ch}.     
	}
	\and
    Shoham Letzter\thanks{
        ETH Institute for Theoretical Studies,
        ETH,
        8092 Zurich;
        e-mail: \texttt{shoham.letzter}@\texttt{eth-its.ethz.ch}.
        Research supported by Dr.~Max
        R\"ossler, the Walter Haefner Foundation and by the ETH Zurich Foundation.
    }
    \and
	Benny Sudakov\thanks{
	    Department of Mathematics, 
	    ETH, 
	    8092 Zurich;
	    e-mail: \texttt{benjamin.sudakov}@\texttt{math.ethz.ch}.     
		Research supported in part by SNSF grant 200021-175573.
	}
	\and
	Tuan Tran\thanks{
	    Department of Mathematics, 
	    ETH, 
	    8092 Zurich;
	    e-mail: \texttt{manh.tran}@\texttt{math.ethz.ch}. Research supported by the Humboldt Research Foundation.     
	}
}

\date{}
\maketitle

\begin{abstract}

	\setlength{\parskip}{\medskipamount}
    \setlength{\parindent}{0pt}
    \noindent

	We call a family $\F$ of subsets of $[n]$ 
	$s$-\emph{\sat} if it contains no $s$ pairwise disjoint sets, and
	moreover no set can be added to $\F$ while preserving this property (here
	$[n] = \{1,\ldots,n\}$).

	More than 40 years ago, Erd\H{o}s and Kleitman conjectured that an
	$s$-\sat{} family of subsets of $[n]$ has size at least $(1 -
	2^{-(s-1)})2^n$.  It is easy to show that every $s$-\sat{} family  has
	size at least $\frac{1}{2}\cdot 2^n$, but, as was mentioned by Frankl and
	Tokushige, even obtaining a slightly better bound of $(1/2 + \eps)2^n$,
	for some fixed $\eps > 0$, seems difficult.  In this note, we prove such a
	result, showing that every $s$-\sat{} family of subsets of $[n]$ has size
	at least $(1 - 1/s)2^n$.

	This lower bound is a consequence of a multipartite version of the
	problem, in which we seek a lower bound on $|\F_1| + \ldots + |\F_s|$ where
	$\F_1, \ldots, \F_s$ are families of subsets of $[n]$, 
	such that there are no
	$s$ pairwise disjoint sets, one from each family $\F_i$, and furthermore
	no set can be added to any of the families while preserving this
	property.  We show that $|\F_1| + \ldots + |\F_s| \ge (s-1)\cdot 2^n$,
	which is tight e.g.\ by taking $\F_1$ to be empty, and letting the
	remaining families be the families of all subsets of $[n]$.

\end{abstract}

\section{Introduction}

 \begingroup
    \renewcommand{\thefootnote}{} 
    \footnotetext{ 
    MSC2010 Classification: ~
      	Primary: 05D05 ~ Secondary: 	60C05        
    } 
    \endgroup
	
	In extremal set theory, one studies how large, or how small, a family
	$\F$ can be, if $\F$ consists of subsets of some set and satisfies
	certain restrictions.
	Let $[n] = \{1,\ldots,n\}$, let $\ps$ be the family of all subsets of
	$[n]$ and let $[n]^{(k)}$ be the family of subsets of $[n]$ of size $k$.

	A classical example in the area is the study of \emph{intersecting
	families}. We say that a family $\F$ is \emph{intersecting} if for every
	$A, B \in \F$ we have $A \cap B \neq \emptyset$.
	The following simple proposition, first noted by Erd\H{o}s, Ko and
	Rado \cite{erdos-ko-rado}, gives an upper bound on the size of an
	intersecting family in $\ps$.

	\begin{prop} \label{prop:intersecting}
		Let $\F \subseteq \ps$ be intersecting, then $|\F| \le 2^{n-1}$.
	\end{prop}
	This follows from the observation that for every set $A \ss [n]$ at most
	one of $A$ and $\setcomp{A}$ (where $\setcomp{A} \defeq [n] \sm A$) is in $\F$.
	This bound is tight, which can be seen, e.g., by taking the family of all
	subsets of $[n]$ that contain the element $1$. In fact, there are many
	more extremal examples (see \cite{erdos-hindman}), 
	partly due to the following proposition.
	\begin{prop} \label{prop:max-intersecting}
		Let $\F \subseteq \ps$ be intersecting, then there is an intersecting
		family in $\ps$ of size $2^{n-1}$ that contains $\F$. In other words,
		if $\F \ss \ps$ is a maximal intersecting family, then it has size $2^{n-1}$.
	\end{prop}
	Indeed, suppose that $\F$ is a maximal intersecting family of size
	less than $2^{n-1}$. Then there is a set $A \subseteq [n]$ such that $A,
	\setcomp{A} \notin \F$. By maximality of $\F$, there exist sets $B, C \in
	\F$ such that $A \cap B = \emptyset$ and $\setcomp{A} \cap C = \emptyset$.
	In particular, $B \cap C = \emptyset$, a contradiction.

	There have been numerous extensions and variations of
	\Cref{prop:intersecting}. For example, the study of $t$-intersecting
	families \cite{katona} (where the intersection of every two sets has size
	at least $t$) and $L$-intersecting families \cite{alon-babai-suzuki}
	(where the size of the intersection of every two distinct sets lies in
	some set of integers $L$). Such problems were also studied for
	$k$-uniform families, i.e.\ families that are subsets of $\ups{n}{k}$
	(see e.g.\ \cite{ahlswede-khachatrian} and \cite{ray-chaudhuri-wilson}).
	A famous example is the Erd\H{o}s-Ko-Rado \cite{erdos-ko-rado} theorem
	which states that if $\F \subseteq \ups{n}{k}$ is intersecting, and $n
	\ge 2k$, then $|\F| \le \binom{n-1}{k-1}$, a bound which is again tight
	by taking the families of all sets containing $1$.  Another interesting
	generalisation of \Cref{prop:max-intersecting} looks for the maximum
	measure of an intersecting family under the $p$-biased product measure
	(see \cite{ahlswede-katona,dinur-safra,friedgut,filmus}).  A different
	direction, which was suggested by Simonovits and S\'os
	\cite{Simonovits-sos}, studies the size of intersecting families of
	structured families, such as graphs, permutations and sets of integers
	(see e.g.\ \cite{borg,godsil-meagher}).
	
	Here we are interested in a different extension of
	\Cref{prop:intersecting,prop:max-intersecting}.
	Given $s \ge 2$, we say that a family $\F \ss \ps$ is
	\emph{$s$-\sat{}} if $\F$ contains no $s$ pairwise disjoint sets,
	and furthermore $\F$ is maximal with respect to this property.
	An example for an $s$-\sat{} family is the set of all subsets of $[n]$
	that have a non-empty intersection with $[s-1]$.
	In 1974 Erd\H{o}s and Kleitman \cite{erdos-kleitman} made the following
	conjecture, which states that this
	example is the smallest $s$-\sat{} family in $\ps$.
	\begin{conj}[Erd\H{o}s, Kleitman \cite{erdos-kleitman}]
		\label{conj:main}
		Let $\F \ss \ps$ be $s$-\sat{}. Then $|\F| \ge (1 - 2^{-(s-1)})\cdot 2^n$.
	\end{conj}

	Note that by \Cref{prop:max-intersecting}, \Cref{conj:main} holds for $s =
	2$. Given a family $\F \ss \ps$, define $\comp{\F} = \ps \sm \F$, and
	$\pwcomp{\F} = \{\setcomp{A} : A \in \F\}$.  Then for every $s \ge 2$, if
	$\F \ss \ps$ is $s$-\sat{} then $\dbcomp{\F}$ is intersecting. Indeed, if
	$A \notin \F$ then $\setcomp{A}$ contains $s - 1$ pairwise disjoint
	sets of $\F$, so if $A$ and $B$ are such that $\setcomp{A}$ and
	$\setcomp{B}$ are disjoint, then at least one of $A$ and $B$ is in $\F$,
	as otherwise $\F$ contains $2(s-1) \ge s$ pairwise disjoint sets, a
	contradiction. By \Cref{prop:intersecting}, it follows that if $\F$ is
	$s$-\sat{} then $|\F| \ge 2^{n-1}$. Surprisingly, beyond this trivial
	lower bound, nothing was known. Moreover, Frankl and Tokushige
	\cite{frankl-survey} wrote in their recent survey that obtaining a lower
	bound of $(1/2 + \eps)2^n$, i.e.\ a modest improvement over the trivial
	bound, is a challenging open problem. In this paper we prove such a result.

	\begin{thm} \label{thm:main}
		Let $\F \subseteq \ps$ be $s$-\sat{}, where $s \ge 2$. Then $|\F| \ge
		(1 - 1/s)2^n$.
	\end{thm}

	In fact, \Cref{thm:main} is a corollary of a multipartite version of the
	above problem. A sequence of $s$ families $\F_1, \ldots, \F_s \subseteq
	\ps$ is called \emph{\crdep{}} (see, e.g., \cite{frankl-kupavskii}) if
	there is no choice of $s$ sets $A_i \in \F_i$, for $i \in [s]$, such that
	$A_1, \ldots, A_s$ are pairwise disjoint.  We call a sequence	of $s$
	families $\F_1, \ldots, \F_s$ \emph{\crsat{}} if the sequence is \crdep{}
	and is maximal with respect to this property, i.e.\ the addition of any
	set to any of the families results in a sequence which is not \crdep. Our
	aim here is to obtain a lower bound on the $|\F_1| + \ldots + |\F_s|$.
	Note that if $\F$ is $s$-\sat{} then the sequence given by $\F_1 = \ldots
	= \F_s = \F$ is \crsat{}. Hence, a lower bound on the sum of sizes of a
	\crsat{} sequence of $s$ families implies a lower bound on the size of an
	$s$-\sat{} family.
	
	A simple example of a \crsat{} sequence $\F_1, \ldots, \F_s$
	can be obtained by taking $\F_1$ to be
	empty, and letting all other sets be $\ps$. 
	This construction is a special
	case of a more general family of examples which we believe
	contains all extremal examples; we discuss this in \Cref{sec:conclusion}.
	%Out 
	Our next result shows that this example is indeed a smallest example for a
	\crsat{} sequence. Furthermore, it implies \Cref{thm:main} by taking
	$\F_1 = \ldots = \F_s = \F$.
	\begin{thm} \label{thm:cross}
		Let $\F_1, \ldots, \F_s \ss \ps$ be \crsat{}. Then $|\F_1| + \ldots + |\F_s|
		\ge (s-1)2^n$.
	\end{thm}

	We have two different approaches to this problem, each of which can be
	used to prove \Cref{thm:cross}. As the proofs are short, and
	\Cref{conj:main} is still open, we feel that there is merit in presenting
	both proofs here in hope that they would give rise to further progress on
	\Cref{conj:main}.

	Our first approach makes use of an interesting connection to correlation
	inequalities. Let us start by defining the \emph{disjoint occurrence} of
	two families. Given subsets $A, I \ss [n]$, let 
	\begin{equation*}
		\C(I, A) = \{S \ss [n] : S \cap I = A \cap I \}.
	\end{equation*}
	The \emph{disjoint occurrence} of two families $\A, \B \ss \ps$ is defined by
	\begin{align*}
		\A \boxpr \B \defeq \{A	: \text{$\exists$ \emph{disjoint}
			sets $I,J \ss [n]$ s.t.\ $\C(I,
			A) \subset \A$ and  $\C(J, A) \subset \B$}\}.
	\end{align*}
	Note that when $\A$ and $\B$ are both increasing families (i.e.\ if $A
	\in \A$, and $A \ss B \ss [n]$ then $B \in \A$), 
	$\A \boxpr \B$ is the set of all subsets of $[n]$ which can be written as
	a disjoint union of a set from $\A$ and a set from $\B$. This notion of
	disjoint occurrence appears naturally in the study of percolation. Using it, one
	can express the probability that there are two edge-disjoint paths
	between two sets of vertices in a random subgraph, chosen uniformly at
	random, of a given graph.
		
	Van den Berg and Kesten \cite{bk} proved that $|\A \boxpr \B| \le
	|\A||\B|/2^n$ for increasing families $\A, \B \ss \ps$ and conjectured
	that this inequality should hold for general families. This was
	proved by Reimer \cite{reimer} in a ground breaking paper
	and is currently known as the van den
	Berg-Kesten-Reimer inequality.

	Disjoint occurrence is surprisingly suitable for the study of \sat{}
	families. For example, if $\F$ is $3$-\sat{} then it is easy to see that
	$\F$ is increasing, so $\F \boxpr \F$ is the family of sets that are
	disjoint unions of two sets from $\F$, which is exactly the family
	$\dbcomp{\F}$. This observation alone implies an improved
	lower bound on $|\F|$ using the van den Berg-Kesten-Reimer inequality.
	We obtain a better bound using a variant of this inequality, which was
	first observed by Talagrand \cite{talagrand}, and later played a major
	role in Reimer's proof of the van den Berg-Kesten-Reimer inequality in
	full generality. 

	Our second approach is algebraic: we define a polynomial for each set in
	a certain family related to $\F_1, \ldots, \F_s$, and show that these
	polynomials are linearly independent, thus implying that the family is
	not very large.

\section{The proof}

	Before turning to the first proof of \Cref{thm:cross}, we introduce the
	correlation inequality that we will need. We present its short proof for
	the sake of completeness.	

	\begin{lem}[Talagrand \cite{talagrand}] \label{lem:improved-bk}
		Let $\A, \B \ss \ps$ be increasing families.
		Then $\left|\A \boxpr \B\right| \le \left|\pwcomp{\A} \cap \B\right|$.
	\end{lem}

	\begin{rem}
		Before turning to the proof of \Cref{lem:improved-bk}, we remark that the
		statement of \Cref{lem:improved-bk} holds even without the assumption
		that the families $\A$ and $\B$ are increasing. Furthermore, an equivalent version of this played a major role in Reimer's proof
		\cite{reimer} of the van den Berg-Kesten-Reimer inequality.
	\end{rem}

	\begin{proof}
		We prove the statement by induction on $n$. It is easy to check 
		it for $n = 1$. 
		Let $n > 1$ and suppose that the statement holds for $n - 1$. Given a
		family $\F \ss \ps$, denote by $\F_0$ the family of sets in $\F$ that
		do not contain the element $n$, and let $\F_1 = \{A \ss [n-1] : A
		\cup \{n\} \in \F\}$.
		In particular, $\F_0 \subseteq \F_1 \ss \psn{n-1}$ when $\F$ is an increasing family.

		We have
		\begin{align*}
			\abs{\A \boxpr \B}
			& = \abs{(\A \boxpr \B)_0} + \abs{(\A \boxpr \B)_1} \\
			& = \abs{\A_0 \boxpr \B_0} + \abs{\A_1 \boxpr \B_0} + \abs{\A_0
			\boxpr \B_1}
			- \abs{(\A_1 \boxpr \B_0) \cap (\A_0 \boxpr \B_1)} \\
			& \le \abs{\A_1 \boxpr \B_0} + \abs{\A_0 \boxpr \B_1} \\
			& \le \abs{\pwcomp{\A_1} \cap \B_0} + \abs{\pwcomp{\A_0} \cap
				\B_1} \\
			& = \abs{(\pwcomp{\A} \cap \B)_0} + \abs{(\pwcomp{\A} \cap
				\B)_1} \\
			& = \abs{\pwcomp{\A} \cap \B}, 	
		\end{align*}
		where the first inequality holds because $\A_0 \boxpr \B_0 \ss
		(\A_1 \boxpr \B_0) \cap (\A_0 \boxpr \B_1)$, and the second one follows
		by induction.
	\end{proof}

	We are now ready for the first proof of \Cref{thm:cross}.
	\begin{proof} [First proof of \Cref{thm:cross}]
		Let $\F_1, \ldots, \F_s$ be \crsat{}, where $s \ge 2$.
		Note that 
		\begin{equation} \label{eqn:comp-F-i}
			\dbcomp{\F_i} = 
			\F_1 \boxpr \ldots \boxpr \F_{i-1} \boxpr \F_{i+1} \boxpr \ldots
			\boxpr \F_s.
		\end{equation}
		Indeed, for every $A \notin \F_i$, $\setcomp{A}$ contains a disjoint
		union of sets from $\F_1, \ldots, \F_{i-1}, \F_{i+1}, \ldots, \F_s$
		and, conversely, any $A \in \F_1 \boxpr \ldots \boxpr \F_{i-1} \boxpr
		\F_{i+1} \boxpr \ldots \boxpr \F_s$ cannot be in $\F_i$ by
		cross dependence.
		By \Cref{lem:improved-bk}, the following holds for every $i \ge  2$. 
		\begin{align} \label{eqn:size-comp-F-i}
		\begin{split}
			\abs{\comp{\F_i}}
			& =	\abs{\dbcomp{\F_i}} \\
			& = \abs{\F_1 \boxpr \ldots \boxpr \F_{i-1} \boxpr \F_{i+1}
				\boxpr \ldots \boxpr \F_s} \\
			& \le \abs{(\F_1 \boxpr \ldots \boxpr \F_{i-1}) \cap \pwcomp{(\F_{i+1}
				\boxpr \ldots \boxpr \F_s)}}.
			\end{split}
		\end{align}
		
		Denote $\G_1 = \comp{\F_1}$, and $\G_i = (\F_1 \boxpr \ldots \boxpr
		\F_{i-1}) \cap \pwcomp{(\F_{i+1} \boxpr \ldots \boxpr \F_s)}$ for $i
		\ge 2$.
		\begin{claim} \label{claim:G-i-disjoint}
			$\G_i \cap \G_j = \emptyset$ for $1 \le i < j \le s$.
		\end{claim}
		\begin{proof}
			Indeed, if $i = 1$ then $\G_1 \subseteq \comp{\F_1}$ and $\G_j
			\subseteq \F_1$.  Otherwise, if $A \in \G_i \cap \G_j$ with $2\le i<j$ then $A$ is
			the disjoint union of elements from $\F_1, \ldots, \F_{j-1}$, so in
			particular (as the sets $\F_l$ are increasing) it is the disjoint
			union of elements from $\F_1, \ldots, \F_i$. Furthermore, since $i\ge 2$, $A$ is also
			the complement (with respect to $[n]$) of a disjoint union of sets in
			$\F_{i+1}, \ldots, \F_s$, i.e.\ $\setcomp{A}$ is the disjoint union
			of sets in $\F_{i+1}, \ldots, \F_s$.  But this means that $[n]$ is
			the disjoint union of sets from $\F_1, \ldots, \F_s$, a contradiction
			to the assumption that $\F_1, \ldots, \F_s$ form a \crsat{} sequence.
		\end{proof}
		
		It follows from \eqref{eqn:comp-F-i}, \eqref{eqn:size-comp-F-i} and
		\Cref{claim:G-i-disjoint} that
		\begin{align} \label{eqn:end-proof}
			\begin{split}
			|\F_1| + \ldots + |\F_s| 
			& = s\cdot 2^n - (|\comp{\F_1}| + \ldots + |\comp{\F_s}|) \\
			& \ge s \cdot 2^n - (|\G_1| + \ldots + |\G_s|) \\
			& \ge s \cdot 2^n - 2^n = (s - 1)2^n,
			\end{split}
		\end{align}
		thus completing the proof of \Cref{thm:cross}.
\end{proof}

	Our next approach is algebraic. Before presenting the proof, we introduce
	some definitions and an easy lemma.  Let $n$ be fixed and consider the
	vector space $V$ (over $\R$) of functions from $\{0,1\}^n$ to $\R$. Note
	that this is a vector space of dimension $2^n$.  Given a subset $S \ss
	[n]$, let $P_S : \{0,1\}^n \rightarrow \R$ be defined by $P_S(x) =
	\prod_{i \in S}x_i$, where $x = (x_1, \ldots, x_n)^T \in \{0,1\}^n$, and
	let %$x_S$ 
	$x_S\in \{0,1\}^n$ be defined by $(x_S)_i = 1$ if and only if $i \in S$. The
	following lemma shows that $\{P_S : S \subseteq [n]\}$ is a linearly
	independent set in $V$ (in fact, as $V$ has dimension $2^n$, it is a
	basis).
	\begin{lem} \label{lem:basis}
		The set $\{P_S: S \subseteq [n]\}$ is linearly independent in $V$.
	\end{lem}
	\begin{proof}
		Suppose that $\sum_{S \ss [n]} \alpha_S P_S = 0$, where $\alpha_S \in
		\R$, and not all $\alpha_S$'s are $0$. Let $T$ be a smallest set such
		that $\alpha_T \neq 0$. Note that $P_S(x_T) = 1$ if and only if $S
		\ss T$.  Hence
		$$
			0  = \sum_{S \ss [n]} \alpha_S P_S(x_T) 
			= \sum_{S \ss [n], |S| \le |T|} \alpha_S P_S(x_T) 
		 = \alpha_T,$$ 
		a contradiction to the assumption that $\alpha_T \neq 0$.
		It follows that $\alpha_S = 0$ for every $S \ss [n]$, i.e.\ the
		polynomials $\{P_S(x) : S \ss [n]\}$ are linearly independent, as
		required.
	\end{proof}

	We shall use the inner product on $V$ which is defined by 
	\begin{equation} \label{eqn:inner-product}
		\ang{f,g} = \sum_{x \in \{0,1\}^n} f(x)g(x).
	\end{equation}
	It is easy to check that this is indeed an inner product; in fact, it is
	the standard inner product, if functions are viewed as vectors
	indexed by $\{0,1\}^n$.

	We are now ready for the second proof of \Cref{thm:cross}.
	\begin{proof}[Second proof of \Cref{thm:cross}]
		Let $\F_1, \ldots, \F_s$ be \crsat{}, where $s \ge 2$.
		Given $i$ and $A \in \dbcomp{\F_i}$, recall that by
		\eqref{eqn:comp-F-i}, $A$ can be written as the disjoint union of
		sets from $\F_1, \ldots, \F_{i-1}, \F_{i+1}, \ldots, \F_n$.
		For every such $i$ and $A$, fix a representation 
		\begin{equation} \label{eqn:rep-A}
			A = B \cup C,
		\end{equation}
		where $B$ is a disjoint union of sets from $\F_1, \ldots, \F_{i-1}$
		and $C$ is a disjoint union of sets from $\F_{i+1}, \ldots, \F_{s}$.
		Let
		\begin{equation*}
			Q_{i, A}(x) = \prod_{j \in B} x_j \cdot \prod_{j \in C} (x_j - 1).
		\end{equation*}
		Let $W_i$ be the family of polynomials $Q_{i, A}$, where $i \in [s]$
		and $A \ss \dbcomp{\F_i}$.
		
		We shall show that the sets $W_i$ are pairwise disjoint and that $W_1
		\cup \ldots \cup W_s$ is linearly independent.
		This will follow from the following two claims, which state that each
		$W_i$ is linearly independent and that $W_i$ and $W_j$ are orthogonal
		for distinct $i$ and $j$.

		\begin{claim} \label{claim:W-i-lin-indep}
			$W_i$ is linearly independent for $i \in [s]$.
		\end{claim}
		\begin{proof}
			Suppose that $\sum_{A \in \dbcomp{\F_i}} \, \alpha_A Q_{i, A}
			= 0$, where $\alpha_A \in \R$ and not all $\alpha_A$'s are $0$.
			Let $A$ be a largest set such that $\alpha_A \neq 0$.
			Note that for every $A' \in \dbcomp{\F_i}$, $Q_{i, A'}$
			can be written as 
			\begin{equation*}
				Q_{i, A'} = P_{A'} + \sum_{S \subsetneq A'} \beta_{A', S} P_S,
			\end{equation*}
			where the values of $\beta_{A',S}$ depend on the representation of
			$A'$ as in \eqref{eqn:rep-A}.
			Hence, by choice of $A$,
			\begin{align*}
				0\, & = 
				\sum_{A' \in \dbcomp{\F_i},\, |A'| \le |A|} \alpha_{A'}
				Q_{i, A'} \\
				& = \sum_{A' \in \dbcomp{\F_i},\, |A'| \le |A|} 
				\alpha_{A'}(P_{A'} + \sum_{S \subsetneq A'}\beta_{A', S} P_S) \\
				& = \alpha_A P_A + \sum_{|S| \le |A|,\, S \neq A} \gamma_S P_S,
			\end{align*}
			for some $\gamma_S \in \R$.
			However, since the $P_S$'s are linearly independent (by
			\Cref{lem:basis}), we have $\alpha_A = 0$, a contradiction.
			It follows that $W_i$ is linearly independent, as required.
		\end{proof}

		\begin{claim} \label{claim:W-i-orthogonal}
			$W_i$ and $W_j$ are orthogonal for $1 \le i < j \le s$.
		\end{claim}
		\begin{proof}
			Let $A \in \dbcomp{\F_i}$ and $A' \in \dbcomp{\F_j}$, where $1
			\le i < j \le s$.  Write $A = B \cup C$ and $A' = B' \cup C'$ for
			the representations as in \eqref{eqn:rep-A}.  Let $x \in
			\{0,1\}^n$.  We claim that $Q_{i, A}(x) = 0$ or $Q_{j, A'}(x) =
			0$.  Indeed, if the former does not hold, then $x_i = 1$ for $i
			\in B$ and $x_i = 0$ for $i \in C$.  Note that $B' \cap C \neq
			\emptyset$, because $\{\F_1, \ldots, \F_s\}$ is \crdep{}.  Hence,
			$x_i = 0$ for some $i \in B'$, which implies that $Q_{i, A'}(x) =
			0$, as claimed.  It easily follows that $\ang{Q_{i,A}, Q_{j, A'}}
			= 0$ (recall the definition of the inner product given in
			\eqref{eqn:inner-product}), as required.
		\end{proof}

		It follows from \Cref{claim:W-i-lin-indep,claim:W-i-orthogonal} that
		$W_1 \cup \ldots \cup W_s$ is linearly independent, hence it has size
		at most the dimension of $V$, i.e.\ at most $2^n$. But $|W_i| =
		|\comp{\F_i}|$, thus, as in \eqref{eqn:end-proof}
		\begin{equation*}
			|\F_1| + \ldots + |\F_s| \ge (s-1)2^n,
		\end{equation*}
		as desired.
	\end{proof}

\section{Conclusion} \label{sec:conclusion}
	There are two main directions for further research that we would like to
	mention here.

	The first is related to the tightness of \Cref{thm:cross}.  As mentioned
	in the introduction, the result is tight, which can be seen by taking
	$\F_1 = \emptyset$ and $\F_2 = \ldots = \F_s = \ps$.  In fact, this is a
	special case of the following class of examples: let $\F_1$ be any
	increasing family in $\ps$, let $\F_2 = \dbcomp{\F_1}$ and let $\F_3 =
	\ldots = \F_s = \ps$. Then $|\F_1|+|\F_2|=2^n$ and it is easy to check
	that any set in $\F_1$ intersect every set in $\F_2$. Therefore, every
	such example yields a \crsat{} set of smallest size. Furthermore, it is
	easy to see that these are  the only examples for which $\F_3 = \ldots =
	\F_s$. It seems plausible that these are the only possible examples (up
	to permuting the order of the families). This problem of classifying all
	extremal examples, interesting in its own right, may give a hint on how
	to further improve the lower bound of the size of $s$-\sat{} families.

	The second, and seemingly more challenging direction, is to improve on
	\Cref{thm:main}. We proved that if $\F$ is $s$-\sat{} then $|\F| \ge
	(1 - 1/s)2^n$, where the conjectured bound is $\left(1 -
	2^{-(s-1)}\right)2^n$. 
	We note that it is possible to improve the lower bound slightly, to show
	that $|\F| \ge \left(1 - 1/s + \Omega(\log n/n)\right)2^n$, by running
	the argument of the first proof more carefully in the case where $\F_1 =
	\ldots = \F_s = \F$; we omit further details.
	It would be very interesting to obtain an improvement of error term $1/s$
	to an expression exponential in $s$. We hope
	that our methods can be used to make further progress on this old
	conjecture.

	Let us mention here a general class of examples of $s$-\sat{} families
	whose size is $\left(1 - 2^{-(s-1)}\right)2^n$.  We do not know of any
	other examples of $s$-\sat{} families, and feel that it is likely that if
	the conjecture holds, then these are the only extremal examples.
	
	\begin{ex}
		Given $s \ge 2$, let $\{I_1, \ldots, I_{s-1}\}$ be a partition of
		$[n]$. For each $i \in [s-1]$, pick a maximal intersecting family
		$\F_i$ of subsets of  $I_i$; in particular, by \Cref{prop:max-intersecting},
		$|\F_i| = 2^{|I_i|-1}$. Define $\F$ as follows.
		\begin{equation*}
			\F = \{A \ss [n] : A \cap I_i \in \F_i \text{ for some $i \in
			[s-1]$.} \}
		\end{equation*}
		It is easy to check that $\F$ is $s$-\sat{} as a family of subsets of
		$[n]$ and that it has size $\left(1 - 2^{-(s-1)}\right) 2^n$.

		Note that this class of examples contains the example that was
		mentioned earlier, of the family of subsets of $[n]$ that intersect
		$[s-1]$.
	\end{ex}

	Finally, we note the following interesting phenomenon.
	\begin{prop}
		If \Cref{conj:main} holds for $s+1$, then it holds for $s$.
	\end{prop}
	Indeed, suppose that \Cref{conj:main} holds for $s+1$, and let $\F \ss
	\ps$ be $s$-\sat{}. Define $\G \ss \psn{n+1}$ as follows.
	\begin{equation*}
		\G = \F \cup \{A \ss [n+1] : n + 1 \in A\}.
	\end{equation*}
	Note that $\G$ is $(s+1)$-\sat{} (as a subset of $\psn{n+1}$). Hence,
	by the assumption that the conjecture holds for $s+1$, we find that $|\G|
	\ge \left(1 - 2^{-s}\right)2^{n+1}$. Note also that $|\G| = |\F| + 2^n$.
	It follows that $|\F| \ge \left(1 - 2^{-s}\right)2^{n+1} - 2^n = \left(1
	- 2^{-(s-1)}\right)2^n$, as required.

	\bibliography{saturated}

\providecommand{\bysame}{\leavevmode\hbox to3em{\hrulefill}\thinspace}
\providecommand{\MR}{\relax\ifhmode\unskip\space\fi MR }
\providecommand{\MRhref}[2]{%
  \href{http://www.ams.org/mathscinet-getitem?mr=#1}{#2}
}
\providecommand{\href}[2]{#2}
\begin{thebibliography}{10}

\bibitem{ahlswede-katona}
R.~Ahlswede and G.~Katona, \emph{Contributions to the geometry of hamming
  spaces}, Discr. Math. \textbf{17} (1977), 1--22.

\bibitem{ahlswede-khachatrian}
R.~Ahlswede and L.~H. Khachatrian, \emph{A pushing-pulling method: new proofs
  of intersection theorems}, Combinatorica \textbf{19} (1999), 1--15.

\bibitem{alon-babai-suzuki}
N.~Alon, L.~Babai, and H.~Suzuki, \emph{{Multilinear polynomials and
  Frankl-Ray-Chaudhuri-Wilson type intersection theorems}}, J. Combin. Theory
  Ser. A. \textbf{58} (1991), 165--180.

\bibitem{borg}
P.~Borg, \emph{Intersecting families of sets and permutations: a survey}, Int.
  J. Math. Game Theory Algebra \textbf{21} (2012), 543--559.

\bibitem{dinur-safra}
I.~Dinur and S.~Safra, \emph{On the hardness of approximating minimum vertex
  cover}, Ann. Math. \textbf{162} (2005), 439--485.

\bibitem{erdos-hindman}
P.~Erd\H{o}s and N.~Hindman, \emph{Enumeration of intersecting families},
  Discr. Math. \textbf{48} (1984), 61--65.

\bibitem{erdos-kleitman}
P.~Erd\H{o}s and D.~J. Kleitman, \emph{Extremal problems among subsets of a
  set}, Discr. Math. \textbf{8} (1974), 281--194.

\bibitem{erdos-ko-rado}
P.~Erd\H{o}s, C.~Ko, and R.~Rado, \emph{Intersection theorems for systems of
  finite sets}, Quart. J. Math. Oxford \textbf{12} (1961), 313--320.

\bibitem{filmus}
Y.~Filmus, \emph{The weighted complete intersection theorem,}, J. Combin.
  Theory A \textbf{151} (2017), 84--101.

\bibitem{frankl-kupavskii}
P.~Frankl and A.~Kupavskii, \emph{Two problems of {P. Erd\H{o}s} on matchings
  in set families}, arXiv:1607.06126, preprint.

\bibitem{frankl-survey}
P.~Frankl and N.~Tokushige, \emph{Invitation to intersection problems for
  finite sets}, J. Combin. Theory, Ser. A \textbf{144} (2016), 157--211.

\bibitem{friedgut}
E.~Friedgut, \emph{On the measure of intersecting families, uniqueness and
  stability}, Combinatorica \textbf{28} (2008), 503--528.

\bibitem{godsil-meagher}
C.~Godsil and K.~Meagher, \emph{{Erd\H{o}s-Ko-Rado Theorems: Algebraic
  Approaches}}, Cambridge University Press, 2016.

\bibitem{katona}
G.~Katona, \emph{Intersection theorems for systems of finite sets}, Acta Math.
  Acad. Sci. Hungar. \textbf{15} (1964), 329--337.

\bibitem{ray-chaudhuri-wilson}
D.~K. Ray-Chaudhuri and R.~M. Wilson, \emph{On $t$-designs}, Osaka J. Math.
  \textbf{12} (1975), 737--744.

\bibitem{reimer}
D.~Reimer, \emph{{Proof of the van den Berg-Kesten conjecture}}, Combin.
  Probab. Comput. \textbf{9} (2000), 27--32.

\bibitem{Simonovits-sos}
M.~Simonovits and V.~S\'os, \emph{Intersection theorems on structures}, Ann.
  Discr. Math. \textbf{6} (1980), 301--313.

\bibitem{talagrand}
M.~Talagrand, \emph{{Some Remarks on the Berg-Kesten Inequality}},
  pp.~293--297, Birkh{\"a}user Boston, Boston, MA, 1994.

\bibitem{bk}
J.~van~den Berg and H.~Kesten, \emph{Inequalities with applications to
  percolation and reliability}, J. Appl. Probab. \textbf{22} (1985), 556--569.

\end{thebibliography}
	\bibliographystyle{amsplain}
\end{document}